\documentclass{amsart}
\usepackage{amsmath, amsthm}
\usepackage{mathtools}
\usepackage{amsfonts}
\usepackage{amssymb}
\usepackage{natbib}
\usepackage{mathtools}
\mathtoolsset{showonlyrefs}
\usepackage{xcolor}
\usepackage{array,multirow,makecell,tabularx}
\usepackage{geometry}
\usepackage[colorlinks, linkcolor=blue,  citecolor=blue, urlcolor=blue]{hyperref}

\newtheorem{theorem}{Theorem}[section]

\newtheorem{lemma}[theorem]{Lemma}
\newtheorem{proposition}[theorem]{Proposition}

\newtheorem{definition}[theorem]{Definition}
\newtheorem{remark}[theorem]{Remark}

\DeclareMathOperator*{\argmin}{argmin}
\DeclareMathOperator*{\Var}{Var}

\newcommand{\minx}{x_*}
\newcommand{\minf}{f^*}

\begin{document}

\title[Accelerated stochastic gradient descent]{Nesterov's method with decreasing learning rate leads to accelerated stochastic gradient descent}
\author{Maxime Laborde and Adam M. Oberman }
\thanks{This material is based upon work supported by the Air Force Office of Scientific Research under award number FA9550-18-1-0167 (A.O.) and by a grant from the Innovative Ideas Program of the Healthy Brains and Healthy Lives initiative (HBHL), McGill University (M.L. and A.O)}

\begin{abstract}
We present a coupled system of ODEs which, when discretized with a constant time step/learning rate, recovers Nesterov's accelerated gradient descent algorithm.  The same ODEs, when discretized with a decreasing learning rate, leads to novel stochastic gradient descent (SGD) algorithms, one in the convex and a second in the strongly convex case.  In the strongly convex case, we obtain an algorithm superficially similar to momentum SGD, but with additional terms.  In the convex case, we obtain an algorithm with a novel order $k^{3/4}$ learning rate.   We prove, extending the Lyapunov function approach from the full gradient case to the stochastic case, that the algorithms converge at the optimal rate for the last iterate of SGD, with rate constants which are better than previously available.
\end{abstract}
\maketitle
\section{Introduction}

%


Recently~\citet{su2014differential} showed that Nesterov's method for accelerated gradient descent can be obtained as the discretization of an Ordinary Differential Equation (ODE).   The work of \cite{su2014differential} resulted in a renewed interest in the continuous time approach to first order optimization, for example \cite{wibisono2016variational, wilson2016lyapunov, wilson2019acceleration} with th e objective of (i) developing new insights into accelerated algorithms, and (ii) obtaining new optimization algorithms.  However, so far there has limited progress made on the second objective.  

We propose a new derivation of Nesterov's method as a coupled system of two first order ODEs: the first is gradient descent with the usual learning rate, and the second is a gradient descent with a faster learning rate.  The second equation is unstable, but when coupled to the first equation, and discretizing with a constant learning rate, we obtain Nesterov's method.  Moreover, this system of ODEs decreases a Lyapunov function. In the case of stochastic gradients, choosing a decreasing learning rate for the same system, leads to an accelerated stochastic gradient descent algorithm.  We prove, using a Lyapunov function approach,  that for both the convex and strongly convex cases,By discretizing a specific choice of ODE, we show that: (i) with a constant learning rate, we obtain Nesterov's method, (ii) with a decreasing learning rate, we obtain an accelerated stochastic gradient descent algorithm.

\subsubsection*{{Notation:}} 
Let $f$ be a proper convex function, and write  $\minx = \argmin_x f(x)$ and $\minf = f(\minx)$.  We say $f$ is $L$-smooth if 
\begin{align}
	\label{Basic_Lsmooth}
	f(y) - f(x) + \nabla f(x)\cdot(x-y) &\leq \frac L 2 |x-y|^2, 
\end{align} 
and $f$ is $\mu$-strongly convex if
\begin{equation}\label{f_mu_strongly_convex_gen}
	f(x) + \nabla f(x)\cdot(y - x) \leq f(y) -  \frac{\mu}{2}|x-y|^2,
\end{equation}
for all $x,y$ respectively.  Write $C_f := \frac{L}{\mu}$ for the condition number of $f$.

\subsection{Stochastic gradient descent}
Stochastic gradient descent (SGD) is the iteration
\begin{equation}
	\label{tempSGD}
	x_{k+1} = x_k - h_k \widehat g(x_k)
\end{equation}
where $h_k$ is the learning rate, and $\widehat g(x_k)$ is an approximation to the gradient  $\nabla f(x_k)$. We study the stochastic approximation setting, meaning that we assume $\widehat g(x,\xi)$ is an unbiased estimator of the gradient $\nabla f(x)$, so that $\mathbb{E}[\widehat{g}(x,\xi)] = \nabla f(x)$, and that the variance is bounded $
\Var(\widehat{g}(x,\xi)) \leq \sigma^2$.  Write 
\begin{equation} \label{perturbed_gradient}
\widehat g(x,\xi)  = \nabla f(x) + e(x, \xi),
\end{equation}
where 
\begin{equation}
\label{ass: mean and varience error}
\mathbb{E}[e] = 0 \, \text{  and  } \, \Var(e) \leq \sigma^2.
\end{equation}
 More general assumptions are discussed in \cite{bottou2018optimization}.  The assumption \eqref{ass: mean and varience error} allows us to express our convergence rates in terms of $\sigma^2$.   Relaxing the assumption \eqref{ass: mean and varience error} to cover the finite sum case is the subject of future work.  
The following fundamental inequality can be found in \cite{bottou2018optimization},
\begin{equation}
\label{eq: bottou inequality}
\mathbb{E}[ f(x_k) -f^*] \leq \left(1 - h \mu  \right )^k (f(x_0) -f^*) + {h C_f} \frac{ \sigma^2}{2}, \quad h \leq \frac 1 L 
\end{equation}
for a constant learning rate, $h ={\alpha}/{L}$.  In particular, with $h$ fixed, the function values converge exponentially, with rate $1 - h\mu$, to within $\frac{hL\sigma^2}{2\mu}$ of the optimal value. In Proposition \ref{prop:convergence neighborhood constant learning rate} we establish an accelerated version of~\eqref{eq: bottou inequality}: for constant learning rate  $h$, 
\begin{equation}
\label{eq: accelerated bottou inequality}
\mathbb{E}[ f(x_k) -f^*] \leq \left (1 - h \sqrt \mu   \right )^k E_0^{SC} 
+   h \frac{\sigma^2}{\sqrt{\mu}},  \quad h \leq \frac 1 {\sqrt{L}}\end{equation}
where $E_0^{SC}$ is the initial value of a Lyapunov function, which by \eqref{E0_init}  below, is bounded by $2(f(x_0)-f^*)$. The inequality \eqref{eq: accelerated bottou inequality} shows that the exponential convergence rate of the previous results is accelerated, in the sense that the rate is now $1 - h\sqrt{\mu}$.   In addition, the neighborhood reached in \eqref{eq: accelerated bottou inequality} is smaller than the one reached in \eqref{eq: bottou inequality}. Other papers have studied accelerated versions of \eqref{eq: bottou inequality} using different methods, see for example \cite{Aybat2019,rabbat2020}. 

Both \eqref{eq: accelerated bottou inequality} and \eqref{eq: bottou inequality}, show that the convergence is improve by sending either the learning rate, $h$, or the variance, $\sigma^2$ to zero.  Variance reduction algorithms such as Stochastic Variance Reduction Gradient (SVRG) \cite{SVRG_2013}, Stochastic Average Gradient (SAG) \cite{SAG_2017} and SAGA \cite{SAGA_2014} aim at reducing the variance.  These methods consist in defining at each step a stochastic gradient, $\widehat{g}_k$, such that its variance converges to zero. However, in many modern applications, variance reduction is not effective~\cite{Defazio_Bottou_2019}.

\subsection{Improved convergence rate for the last iterate of SGD}
In this work, we focus on convergence rates for the last iterate of SGD, using a decreasing learning rate.


Convergence results for averages of previous iterates are available in~\cite{Polyak_average_1992, Rakhkin2012, lacoste2012simpler, shamir2013stochastic}, for instance.  In \cite{AWBR_2009}, the authors proved that the optimal bound for the last iterate are $\mathcal{O}\left({1}/{\sqrt{k}} \right)$ (convex case) and $\mathcal{O}\left({1}/{k} \right)$ (strongly convex case).   Because the convergence rates are slow for SGD, improving the rate constant can make a difference: for the $1/k$ rate improving the constant by a factor of ten means the same error bound is achieved in one tenth the number of iterations.  In the rate constants which follow, $G^2$, is a bound for $\mathbb{E}[ \widehat g^2 ]$. 

 In the strongly convex case, we obtain the optimal $\mathcal{O}(1/k)$ rate for the last iterate, but improve the rate constant, be removing the dependance of the rate on the $L$-smoothness bound of the gradient.  (The constant, $L$, appears in the algorithm, to initialize the  learning rate, but not in the convergence rate).  Previous results,  \cite{nemirovski2009robust, shamir2013stochastic, jain2018accelerating}, have a rate which depends on $G$. See Table~\ref{fig: rates in expectation str convex}.
 
 In the convex case, the optimal rate for the last iterate of SGD (see \cite{shamir2013stochastic}) is order $\log(k)/\sqrt{k}$, with a rate constant that depends on $G^2$.  \cite{JainNN19} remove the log factor, by assuming that the number of iterations is decided in advance.   We obtain the $\mathcal{O}(\log(k)/\sqrt{k})$ rate for the last iterate, but again with removing the factor $G$ and the dependence on the $L$-smoothness.  See Table~\ref{fig: rates in expectation convex}. 	
 	
\begin{table*}[h!]
\caption{Convergence rate of $\mathbb{E}[f(x_k) -f^*]$ for a convex, $L$-smooth function.  $G^2$ is a bound on $\mathbb{E}[\widehat{g}^2 ]$, and $\sigma^2$ is a bound on the variance of $ \widehat{g}- \nabla f$. $h_k$ is the learning rate, $E_0$ is the initial value of the Lyapunov function. Top: $\mu$-strongly convex case, $h_k = \mathcal O (1/k)$. Bottom: convex case,   $D$ is the diameter of the domain,  $c$ is a free parameter.}
\label{fig: rates in expectation convex}
\label{fig: rates in expectation str convex}
\begin{center}
\begin{tabular}{|c|c|c|c|} 
\hline
 \cite{nemirovski2009robust}& \cite{shamir2013stochastic} &   \cite{JainNN19}  &  Acc. SGD (Prop~\ref{prop: decrease expectation E_k str convex acc}) \\ 
\hline 
 $ \dfrac{2C_f G^2}{\mu k }$ & $ \dfrac{17 G^2(1+ \log(k))}{\mu k }$ & $ \dfrac{130 G^2}{\mu k }$ &  $ \dfrac{4\sigma^2}{\mu k + \frac{4\sigma^2}{E_0^{SC}}}$ \\
\hline
\end{tabular}
\vspace{.2cm}
\begin{tabular}{|c|c|c|c|} 
\hline
 &\cite{shamir2013stochastic}&   Acc. SGD (Prop~\ref{prop: decrease expectation E_k convex acc}) \\ 
\hline 
$h_k$ & $\dfrac{c^2}{\sqrt{k}}$  &  $\dfrac{c}{ (k+1)^{3/4}}$\\  
\hline
 Rate &$\left(\dfrac{D^2}{c^2} + c^2 G^2 \right)\dfrac{(2+ \log(k))}{ \sqrt{k} }$
 & 
$ \dfrac{ \frac{E_0^C}{16c^2} + c^2 \sigma^2(1 +\log(k))}{\sqrt{k}}$
 \\
\hline
\end{tabular}
\end{center}

\end{table*}

\subsection{Nesterov's methods and Stochastic Gradient Descent}
Nesterov's method \cite{nesterov1983method} for accelerated gradient descent  can be written as 
\begin{align}
\label{genericN1}\tag{AGD1}
	x_{k+1} &= y_k - \frac 1 L \nabla f(y_k)\\
\label{genericN2}\tag{AGD2}
	y_{k+1} &= x_{k+1} + \beta_k(x_{k+1}-x_k)
\end{align}
 The general formulation of Nesterov's method involves a third variable, $v$, the special case above arises when that variable has been eliminated.   With the correct choice of parameters, this algorithm has the optimal convergence rate amongst first order (gradient-based) algorithms for convex optimization. Observe that taking $\beta = 0$ recovers standard gradient descent, and reduces the system to a single variable, $x$. 
The optimal parameters correspond to 
\begin{equation}
	\label{AGD_SC_params}
	\tag{$\beta$-SC}
\beta = \beta_k = \frac{\sqrt{C_f} - 1}{\sqrt{C_f} + 1}.
\end{equation}
in the $\mu$-strongly convex case, and 
\begin{equation}
	\label{AGD_Convex_params}
		\tag{$\beta$-C}
\beta_k = \frac{k}{k+3}
\end{equation}
in the convex case. 

The accelerated stochastic gradient descent methods {we introduce} can be written as generalizations of  \eqref{genericN1}-\eqref{genericN2}.   Popular heuristic methods for accelerated SGD consist of using the iteration \eqref{genericN1}\eqref{genericN2}, replacing $1/L$ with $h_k$, a decreasing learning rate, see for example~\cite{bengio2013advances} \cite{kingma2014adam}.   In practice, these algorithms converge faster than SGD.  We show that that decreasing learning rate versions of the ODEs which lead to Nesterov's method result in systems of the form
\begin{align}
\label{genericASGD1}\tag{ASGD1}
	x_{k+1} &= y_k - \frac{\alpha_k}{L} \nabla f(y_k)\\
\label{genericASGD2}\tag{ASGD2}
	y_{k+1} &= x_{k+1} + \beta_k(x_{k+1}-x_k) + \gamma_k (y_k - x_k)
\end{align}
where $\alpha_k,\beta_k$, and $\gamma_k$ are parameters.  The additional parameter, $\gamma$, does not appear in heuristic methods based on \eqref{genericN1}-\eqref{genericN2}. We establish improved convergence rates for this algorithm, using the coefficients provided, in the convex and strongly convex case.  The result can be interpreted as justification for the empirically observed results for heuristic momentum SGD, in the sense that a similar algorithm has a proven accelerated convergence rate. 

In the strongly convex case,  the coefficients are given, in  Proposition~\ref{prop: generalization Nesterov SC} below, by
\[
\alpha_k = \frac {\sqrt{C_f}}{ \sqrt{C_f} + k/2}, \qquad \beta_{k} = \frac{\sqrt{C_f} - 1}{\sqrt{C_f} + 1 + (k+1)/2 }, \qquad \gamma_k = \frac{k }{2\sqrt{C_f} +k + 3},
\]  
Notice that the algorithm generalizes~\eqref{AGD_SC_params} in the sense that when $k=0$, we recover~\eqref{AGD_SC_params}.  

For comparison,  write SGD as $x_{k+1} = x_k - \alpha_k \widehat g(x_k)/L$.  Then a typical learning rate schedule for SGD written is given by
\[
 \qquad  \alpha_k = \frac {C_f}{ k +C_f}
\]

In the convex case, Proposition \ref{prop: generealization Nesterov C} leads to the following coefficients
\[
\alpha_k =\sqrt{L} h_k,
	\qquad
	\beta_k = \frac{h_{k+1}}{t_{k+1}}(t_k \sqrt{L} -1)  
		\qquad
	\gamma_k = \frac{t_k}{t_{k+1}} \frac{h_{k+1}}{h_k}(1 - \sqrt{L}h_k) 
\] 
where $h_k = \frac{c}{(k+1)^{3/4}}$ and $t_k = \sum_{i=0}^k h_i$.  
which also generalizes the convex coefficients \eqref{AGD_Convex_params}, when $h_k = \frac{1}{\sqrt{L}}$ and $t_k = h_k(k+2)$.

\subsection{Outline of our approach}    
~

Stochastic gradient descent can be interpreted as a time discretization of the gradient descent ODE, $\dot x(t) = -\nabla f(x(t))$, where $\nabla f(x)$ is replaced, in the stochastic approximation setting, by perturbed gradients $\widehat g = \nabla f +e$. It is well-known that solutions of the ODE decrease a rate-generating Lyapunov function. In the perturbed case, an error term will appear in the dissipation of the Lyapunov function. This error can be handle choosing a specific schedule for the time step/learning rate, \cite{bottou2016optimization,Oberman_Prazeres}. 

In what follows, we study accelerated methods for stochastic gradient descent.  
The analysis we perform comes from writing the accelerated algorithm using the three variable formulation.  We present a derivation of Nesterov's method as a coupling between two gradient descent evolutions, where the first evolution has the usual learning rate, and where the second on has a larger rate. We extend this analysis to a system of ODEs for accelerated gradient descent, which also has an associated Lyapunov function.  The accelerated SGD algorithm can be interpreted as a discretization of the ODE system for Nesterov's method, but with a decreasing learning rate.

Thus the analysis of accelerated SGD can be seen as a generalization of the Liapunov function approach to gradient descent, but with the gradient descent ODE replaced by the accelerated gradient descent ODE.  For us, the advantage of the approach is that it allows us the reuse proofs available in the deterministic gradient setting.  Breaking down the steps in this manner is important, because the The Lyapunov analysis in the accelerated case is substantially more complex than in the gradient descent case.   For example, the proofs in~\cite{wilson2016lyapunov} are quite long, and not easy to follow.  However, our proof of \autoref{thm:Acc GD SC} in the stochastic gradient case is shorter than the corresponding result in \cite{wilson2016lyapunov}, yet it generalizes the result.

The key step, which we explain in the strongly convex case, is to obtain the following dissipation of a  rate-generating Lyapunov function $E_k$ after applying the algorithm for a single step, with learning rate $h_k$,  
\begin{equation}\label{inequ_main}
E_{k+1}
 \leq (1-\sqrt{\mu}h_k) E_{k+1}  + h_k\beta_k.
\end{equation}
Here, $h_k$ is the learning rate and the term  $\beta_k$ depends on the error $e_k$ from \eqref{perturbed_gradient}, which reduces to zero in the full gradient case $e =  0$. 
When we take expectations in \eqref{inequ_main}, $\mathbb{E}[\beta_k]$ is proportional to  $h_k \sigma^2$ (see Proposition \ref{thm:Acc GD SC})
\[
\mathbb{E}[E_{k+1}]
 \leq (1-\sqrt{\mu}h_k) \mathbb{E}[E_{k}]  + A_k h_k^2 \sigma^2,
\]
where $A_k$ is a constant which may depend on $k$. Choosing a particular $h_k$ allows us to sum up the error terms and obtain the accelerated rate of convergence presented in Table~\ref{fig: rates in expectation str convex}.
We say the algorithm accelerates SGD in the sense that the term $\mu h_k$ in~\eqref{eq: bottou inequality} is replaced by $\sqrt \mu h_k$. 


\cite{SLRB11} have done related work on accelerated SGD, but it applies in the case where 
 the magnitude of the errors decrease quickly enough to obtain the same rates as in the deterministic case.  A continuous time version of the analysis from \cite{SLRB11} was performed in \cite{attouch2016fast}. 

\subsection{Other related work}
 Continuous time analysis also appears in \cite{flammarion2015averaging}, \cite{lessard2016analysis}, and  \cite{NIPS2015_5843}, among many other recent works. The Lyapunov approach to proving convergence rates appears widely in optimization, for example, see \cite{beck2009fast} for FISTA. 

Convergence rates for averaged SGD are available in a wide setting, (\cite{bottou2016optimization}, \cite{lacoste2012simpler,Rakhkin2012,qian2019sgd}). In the non-asymptotic regime, the last iterate of SGD is often preferred in current applications.  The optimal convergence rate for SGD  is $O(1/k)$ in the smooth, strongly convex case, and $O(1/\sqrt{k})$ in the convex, nonsmooth case~(\cite{nemirovski2009robust}, \cite{bubeck2014optimization}).

When SGD is combined with momentum  (\cite{polyak1964some,nesterov2013introductory}) empirical performance is improved, but this improvement is not always theoretically established~(\cite{kidambi2018insufficiency}).  
 Accelerated versions of  stochastic gradient descent algorithms are comparatively more recent: they appear in \cite{lin2015universal} as well as in \cite{frostig2015regularizing} and \cite{jain2018accelerating}.  A direct acceleration method with a connection to Nesterov's method can be found in~\cite{allen2017katyusha}. \\

\paragraph{\textbf{Organization:}} 
We organize the paper as follows. In Section \ref{section: derivation Nest SC}, we derive our Accelerated Stochastic Gradient Descent and in Section \ref{section: accelerated str convex}, we study the rate convergence  of the last iterate of this scheme based on a Lyapunov analysis, in the strongly convex case. Then, in Sections \ref{section: accelerated convex} and \ref{section: derivation Nest C}, we study the convex case, first deriving a new algorithm and, then, studying the dissipation of a Lyapunov function. Finally, in the last Section \ref{section: simulations}, we present numerical simulations.

%

\section{Derivation A-GD and A-SGD: Strongly convex case}
\label{section: derivation Nest SC}

\subsection{Derivation of accelerated algorithm}
In this section we present a motivating derivation of Nesterov's accelerated gradient descent in the strongly convex case.

The starting point of our analysis is to consider standard gradient descent for $\mu$-strongly convex, $L$-smooth  functions. 
\begin{equation}
\label{eq: Gradient Descent} \tag{GD}
x_{k+1} = x_k - \frac{1}{L}\nabla f(x_k).
\end{equation}
However, gradient descent  converges slowly when the condition number, $C_f$, is large. So we consider 
an auxiliary variable, $v$, which undergoes gradient descent with a multiplier~$\sqrt{C_f}$.
\begin{equation}
\label{eq: unstable Gradient Descent} \tag{U-GD}
v_{k+1} = v_k - \sqrt{C_f} \frac{1}{L} \nabla f(v_k).
\end{equation}
If $C_f > 1$, the $v$ equation becomes unstable.
Couple the two equations, by introducing a forcing which pushes $x$ and $v$ together, with multiplier $w$.
Finally, to avoid the evaluation of the gradient at two different points, we reduce to a single gradient evaluation at $y_k$, given as a convex combination of $x_k$ and $v_k$, with weight $w$.   Together, these considerations lead to the following system involving the variables $x,v$ and $y$.  See Figure~\ref{figureAGD} for an illustration. 
\begin{align}
\label{ySC}
	y_k  & =   \left(1-w\right) x_k  + w v_k	\\
	\label{xprimewSC}
	x_{k+1} - x_k &= w (v_k - x_k) - \frac{1}{L} \nabla f(y_k)\\
	\label{vprimewSC}
    v_{k+1} - v_k &= w (x_k -v_k) -   \frac{\sqrt{C_f}}{L} \nabla f(y_k)
\end{align}
\begin{figure}
\begin{center}
\includegraphics[ scale=0.3]{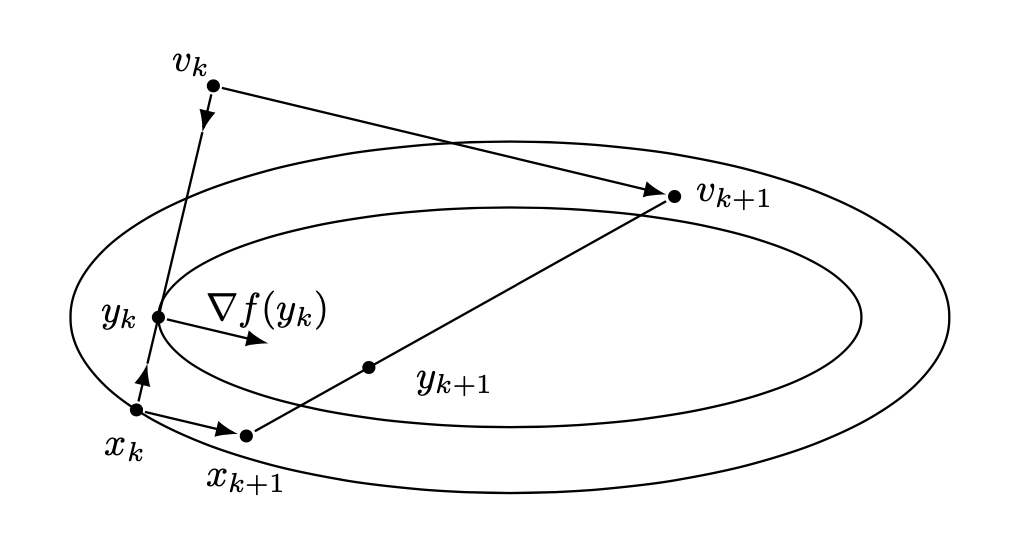}
\end{center}
\caption{Illustration of the role of the variables $x,v,y$ in the algorithm.}
\label{figureAGD}
\end{figure}
Next we show that the system obtained above corresponds to Nesterov's accelerated methods, in the strongly convex case, provided we make the optimal choice of $w$. 
\begin{proposition}
\label{prop: equivalence Nesterov SG}
In the system \eqref{ySC}, \eqref{xprimewSC} \eqref{vprimewSC},
set 
\begin{equation}\label{w_Nest}
	w =\frac{1}{1+\sqrt{C_f}}
\end{equation}
then we obtain Nesterov's accelerated gradient descent \eqref{genericN1} \eqref{genericN2} in the strongly convex case \eqref{AGD_SC_params}. 
\end{proposition}
\begin{proof}
By definition of $y_k$ in \eqref{ySC} we obtain \eqref{genericN1} from \eqref{xprimewSC}.
\[
x_{k+1} = y_k - \frac 1 L  \nabla f(y_k)
\]
Notice that \eqref{w_Nest} leads to 
\begin{equation}\label{ww1}
	\frac{ 1- w}{w} = \sqrt{C_f}
\end{equation}
Next, collect the $v_k$ term in \eqref{vprimewSC} and   and eliminate the gradient term using \eqref{genericN1},  to obtain 
\begin{align*}
v_{k+1} &= (1-w)v_k + wx_k + \sqrt{C_f}(x_{k+1}-y_k)\\
&= 	(1-w)v_k + wx_k + \frac{ 1- w}{w}(x_{k+1}-y_k), && \text{ by \eqref{ww1}}
\end{align*}
Eliminate the variables $v_{k+1}$ and $v_k$ using \eqref{ySC} applied at $k$ and $k+1$, respectively, to obtain 
\begin{align*}
	\frac{1}{w} (y_{k+1} - (1-w)x_{k+1})
 = \frac{1-w}{w}(y_{k} - (1-w_k)x_{k}) + wx_k + \frac{ 1- w}{w} (x_{k+1}-y_k)  
\end{align*} 
Simplify, noticing that the $y_k$ terms cancel, to obtain
\[ 
y_{k+1} = x_{k+1} + (1-2w)(x_{k+1}-x_k) 
\]
Using  \eqref{w_Nest}, the definition of $w$, again leads to \eqref{genericN2} with $(1-2w)= \frac{\sqrt{C_f} - 1}{\sqrt{C_f} + 1}$ giving the value of $\beta$ in \eqref{AGD_SC_params} as desired. 
\end{proof}

\subsection{Extension to the stochastic case: variable learning rate} 
The system \eqref{ySC}, \eqref{xprimewSC} \eqref{vprimewSC}, decreases the Lyapunov function $E^{SC}$, see \eqref{eq:lyapunov SC discrete} below.   In order to generalize to the stochastic case, we first go to a continuous time formulation, then discretize time, to get the algorithm in the time dependent case.  This leads to an equation with the generalized form which follows.

We present the generalization without explanation.  The terms which appear can be justified as either: (i) discretization of the continuous time algorithm, choosing coefficients when $h\to 0$ such that $w/h \to \sqrt{\mu}$   or (ii) taking the Lyapunov function and going through the dissipation proof, choosing the coefficients which make the proof go through.    Alternatively, we could achieve the same results using different coefficients for the algorithms, by changing the coefficient of the quadratic term in the Lyapunov function. \\

Now we go back to the system 
 \eqref{ySC}, \eqref{xprimewSC} \eqref{vprimewSC} and generalize it 
 to allow $w$ to depend on $k$, and to have a learning rate $\alpha_k\leq 1$ 

This leads to the following extension of  \eqref{ySC}, \eqref{xprimewSC} \eqref{vprimewSC}, 
 \begin{align}
 \label{yck}
 	y_k  & =   \left(1-w_k\right) x_k  + w_k v_k	\\
	\label{xpk}
	x_{k+1} - x_k &= w_k (v_k - x_k) - {\alpha_k}\frac {\nabla f(y_k)}{{L}}\\
	\label{vpk}
    v_{k+1} - v_k &= w_k (x_k -v_k) - \sqrt{C_f} {\alpha_k}\frac {\nabla f(y_k)}{{L}}
\end{align}
The following definitions arise in the Lyapunov analysis which follows below. The definition of $w$, \eqref{wkdefnCC1} follows from the dissipation of the Lyapunov function, used to generate the basic inequality, and that the choice of \eqref{hk} will follow from the induction used to determine the rate in the scheduled learning rate case.
\begin{align}
\label{hk}
\alpha_k &= \frac {\sqrt{C_f}}{\sqrt{C_f} + k/2}\\		
\label{wkdefnCC1}
w_k &= \frac{\alpha_k  }{\sqrt {C_f} + \alpha_k}  = \frac 1 {\sqrt{C_f} + 1 + k/2},
\end{align}
Note that the middle expression \eqref{wkdefnCC1} extends \eqref{w_Nest}, since $\alpha_k \leq 1$.

\paragraph*{\textbf{Elimination $v$ variable:}} Here we show how the stochastic algorithm gives a generalization of AGD, by eliminating the $v$ variable, to write the equation in a more standard way. 

\begin{proposition}
\label{prop: generalization Nesterov SC}
The system \eqref{yck}\eqref{xpk}\eqref{vpk} with the choices  \eqref{hk} and \eqref{wkdefnCC1} is equivalent to 
\begin{align}\label{xeqncc}
		x_{k+1} &= y_k - {\alpha_k}  \nabla f(y_k)/L\\
y_{k+1} &= x_{k+1} +   \beta_{k}(x_{k+1} - x_k) + \gamma_k  (y_k - x_k)  
\end{align}
along with 
\begin{equation}
	\label{ASGD_SC_Coeffs}
\beta_{k} = \frac{\sqrt{C_f} - 1}{\sqrt{C_f} + 1 + (k+1)/2 },
\qquad
\gamma_k = \frac{k }{2\sqrt{C_f} +4 +k}
\end{equation}
which is an extension of \eqref{genericN1}\eqref{genericN2} with coefficients \eqref{AGD_SC_params} when $k = 0$.  
\end{proposition}
Note also the new $y_k-x_k$ term which appears. 
\begin{proof}
As in in the constant learning rate case, Prop~\ref{prop: equivalence Nesterov SG}, we obtain~\eqref{xeqncc}.
Next, collect the $v_k$ term in \eqref{vpk} and   and eliminate the gradient term using \eqref{genericN1},  to obtain 
\begin{align*}
v_{k+1} &= (1-w_k)v_k + w_kx_k + \sqrt{C_f}(x_{k+1}-y_k)
\end{align*}
eliminate the variable $v_{k+1}$ and $v_k$  using  using \eqref{yck} applied at $k$ and $k+1$, to obtain 
\begin{align*}
	\frac{1}{w_{k+1}} (y_{k+1} - (1-w_{k+1})x_{k+1})
 = \frac{1-w_k}{w_k}(y_{k} - (1-w_k)x_{k}) + wx_k + \sqrt{C_f}(x_{k+1}-y_k)  
\end{align*} 
which simplifies to 
\[ 
y_{k+1} = x_{k+1} +  w_{k+1} 
\left [ (1/w_k - 1 - \sqrt{C_f})y_k + (\sqrt{C_f}- 1)x_{k+1} + (2-1/w_k)  x_k   \right ]
\]

Simplify, using \eqref{wkdefnCC1}, to obtain 
\[ 
y_{k+1} = x_{k+1} +  w_{k+1} \left [ (\sqrt{C_f} - 1)(x_{k+1} - x_k) + \frac k 2 (y_k - x_k)   \right ]
\]
finally, use \eqref{wkdefnCC1} at $k+1$ and simplify to obtain
the final equation. 
\end{proof}

\subsection{Transition to continuous time}

The system \eqref{yck}\eqref{xpk}\eqref{vpk} with the choice \eqref{wkdefnCC1} can be seen as the forward Euler discretization of the system of ODEs, 
 \begin{equation}\label{NODESC}\tag{1st-ODE-SC}
\begin{cases}
		\dot{x} = \sqrt{\mu} (v-x) -\nabla f(x)/{\sqrt{L}},\\
\dot{v} =\sqrt{\mu} (x-v) -\sqrt{C_f} \nabla f(x)/{\sqrt{L}},\\	\end{cases}
\end{equation}
This system is equivalent to the second order equation with Hessian damping for a smooth~$f$
\begin{equation}\label{OurEqnSCC}\tag{H-ODE-SC}
	\ddot{x} + 2\sqrt{\mu}\dot{x} + \nabla f(x)  = - \frac{1}{\sqrt{L}} \left( D^2 f(x) \cdot \dot{x} + \sqrt{\mu}\nabla f(x) \right).
\end{equation} 
The equation \eqref{OurEqnSCC} can be seen as a combination between Polyak's ODE, \cite{polyak1964some}, 
\begin{equation}\label{polyak_ode}\tag{A-ODE-SC}
	\ddot{x} + 2\sqrt{\mu}\dot{x} + \nabla f(x)  = 0.
\end{equation}
 and the ODE for continuous Newton's method.

Notice that \eqref{OurEqnSCC} is a perturbation of \eqref{polyak_ode} of order $\frac{1}{\sqrt{L}}$, and the perturbation goes to zero as $L\to\infty$.  Similar ODEs have been studied by \cite{alvarez2002second}, they have been shown to accelerate gradient descent in continuous time in \cite{attouch2016fast}.  

There is more than one ODE which can be discretized to obtain Nesterov's method. 
\citet{shi2018understanding} introduced a family of high resolution second order ODEs which also lead to Nesterov's method. In this context, \eqref{OurEqnSCC} corresponds to the high-resolution equation with the parameter $\frac{1}{\sqrt{L}}$.  Making the specific choice of first order system \eqref{NODESC} considerably simplifies the analysis, allowing for shorter, clearer proofs which generalize to the stochastic gradient case (which was not treated in \cite{shi2018understanding}).  \\

In the next section, we are going to generalize the Lyapunov analysis for \eqref{polyak_ode} to the stochastic case using \eqref{NODESC}.

\section{Accelerated method: Strongly Convex case}
\label{section: accelerated str convex}  

In this section, we use a Lyapunov analysis to study the rate of convergence of the generalisation of Nesterov algorithm, \eqref{xeqncc}-\eqref{ASGD_SC_Coeffs} where $\nabla f$ is replaced by the stochastic gradient $\widehat{g}$, \eqref{perturbed_gradient}-\eqref{ass: mean and varience error}. First, we study the dissipation of the Lyapunov function along \eqref{yck}\eqref{xpk}\eqref{vpk} with \eqref{wkdefnCC1} during one step. Then, we highlight two regimes:
\begin{itemize}
\item if the initial value of the Lyapunov function is bigger than a critical value $E_{crit}$, using a contant learning rate, we obtain an accelerated exponential rate to a neighborhood of the minimum,
\item and once $E_{crit}$ is reached, if $h_k$ satisfies \eqref{hk}, then $\mathbb{E}[f(x_k) -f^*]$ converges at the optimal rate $\frac{1}{k}$ with an accelerated rate constant independent of $L$. 
\end{itemize}

\subsection{Dissipation of the Lyapunov function}
\begin{definition}
Define the continuous time Lyapunov function,
\begin{equation}
\label{eq:lyapunov SC discrete}
	E^{SC}(x,v)= f(x) - f^* + \frac{\mu}{2} \|v -x^*\|^2,
\end{equation}
and write $ E_k^{SC} = E^{SC}(x_k,v_k)$.
\end{definition}
Given initial data, $x_0$, we make the convention that $v_0 = x_0$, so that 
\begin{equation}\label{E0_init}
	E_0^{SC} = f(x_0)-f^* +\frac{\mu}{2} |x_0 -x^*|^2  \leq 2(f(x_0)-f^*)
\end{equation}
Replacing gradients $\nabla f$ by $\widehat{g}$ in \eqref{xpk}-\eqref{vpk}
and writing  $h_k = \alpha_k \sqrt{L}$ ,
we obtain
\begin{align}
	\label{Sto_FE-SC-x}
	x_{k+1} - x_k &= w_k (v_k - x_k) - \frac{h_k}{\sqrt{L}}( \nabla f(y_k) +e_k),\\
	\label{Sto_FE-SC-v}	
v_{k+1} - v_k &= w_k (x_k - v_k)-\frac{h_k  }{\sqrt{\mu}}(\nabla f(y_k)	+e_k),	
\end{align}
We set 
\begin{equation}\label{wk_stock}
w_k= \frac{h_k\sqrt{\mu}}{1+h_k\sqrt{\mu}}.
\end{equation}
so that
\begin{equation}\label{wk_inv}
\frac{w_k}{1-w_k}= h_k\sqrt{\mu}
\end{equation}
which is an extension of \eqref{ww1}.
In the next proposition, we establish the dissipation of $E^{SC}$ along \eqref{Sto_FE-SC-x},\eqref{Sto_FE-SC-v}, \eqref{yck} \eqref{wk_stock}. In particular, it proves that in the full gradient case ($e_k = 0$), $E^{SC}$ is, indeed, a Lyapunov function for \eqref{NODESC}. This proposition is a generalization of a result in \cite{wilson2016lyapunov}, for example. 
\begin{proposition}
\label{thm:Acc GD SC}
Let $x_k, v_k$ be two sequences generated by \eqref{Sto_FE-SC-x},\eqref{Sto_FE-SC-v}, \eqref{yck} \eqref{wk_stock} with initial condition $(x_0,v_0)$. Suppose that $h_k\leq \frac{1}{\sqrt{L}}$, then
\begin{equation}
\label{eq: dissipation lip sto str convex}
{E}_{k+1}^{SC} \leq (1- h_k\sqrt{\mu}) {E}_k^{SC} + h_k\beta_k ,
\end{equation}
where 
\[\beta_k =  2h_k \left\langle \nabla f(y_k) + \frac{e_k}{2} ,e_{k} \right\rangle -\left\langle \sqrt{\mu}(x_k - y_k + v_k -x^*) -\frac{1}{\sqrt{L}}\nabla f(y_k),e_k\right\rangle.\]
In particular,
\begin{equation}
\label{eq: main inequality expectation SC}
\mathbb{E}[E_{k+1}^{SC}] \leq (1- h_k\sqrt{\mu}) \mathbb{E}[{E}_k^{SC}] + h_k \sigma^2.
\end{equation}
\end{proposition}

In order to simplify the algebra, we will split the $E^{SC}$ term into $f$ and the quadratic term and prove lemmas on the dissipation of each term.
 Moreover, note that all the terms involving $e_k$ are carried along to obtain the main result.  So this result extends the full gradient without additional estimates on the $e_k$.  The final step will be to take expectations to yield the main inequality \eqref{eq: main inequality expectation SC}.

\begin{lemma}[Dissipation of $f$]
\label{le: dissipation f SC}
Suppose that $f$ is a $\mu$-strongly convex and $L$-smooth function, then
\begin{align*}
f(x_{k+1}) -f(x_k) \leq & \langle \nabla f(y_k) , y_k -x_k \rangle - \frac{\mu}{2}|y_k -x_k|^2 + \left( \frac{h_k^2}{2} - \frac{h_k}{\sqrt{L}}\right) |\nabla f(y_k)|^2\\
& - \frac{h_k}{\sqrt{L}} \langle \nabla f(y_k) , e_k \rangle + h_k^2 \left\langle \nabla f(y_k) +\frac{e_k}{2}, e_k \right\rangle.
\end{align*}
\end{lemma}

\begin{proof}
By adding and subtracting $f(y_k)$, we obtain
\begin{align}
\label{eq: dissip f SC}
f(x_{k+1}) - f(x_k) = & f(x_{k+1})- f(y_k) + f(y_k) - f(x_k)\nonumber\\
\leq & \langle \nabla f(y_k) , x_{k+1} -y_k \rangle + \frac{L}{2}| x_{k+1} -y_k |^2 && \text{by $L$-smoothness of $f$}\\
 & + \langle \nabla f(y_k) , y_{k} -x_k \rangle - \frac{\mu}{2}| y_k -x_k |^2 && \text{by $\mu$-convexity of $f$}\nonumber.
\end{align}

As before, using \eqref{yck},  \eqref{Sto_FE-SC-x} can be rewritten as 
$$ x_{k+1} = y_k - \frac{h_k}{\sqrt{L}}( \nabla f(y_k) +e_k).$$
Then, plugging it into the first line in the right hand side of \eqref{eq: dissip f SC} and expanding the square, we get
\begin{align*}
 \langle \nabla f(y_k) , x_{k+1} -y_k \rangle + \frac{L}{2}| x_{k+1} -y_k |^2 = & - \frac{h_k}{\sqrt{L}} \langle \nabla f(y_k) , \nabla f(y_k) +e_k\rangle + \frac{h_k^2}{2}| \nabla f(y_k) +e_k|^2\\
 = & h_k \left( \frac{h_k}{2} - \frac{1}{\sqrt{L}} \right) |\nabla f(y_k) |^2 \\
 & - \frac{h_k}{\sqrt{L}} \langle \nabla f(y_k) , e_k\rangle + h_k^2 \langle  \nabla f(y_k) +\frac{e_k}{2} , e_k\rangle && \text{terms in $e_k$}
\end{align*}

To conclude the proof, use this expression in \eqref{eq: dissip f SC}. 
\end{proof}

\begin{lemma}[Dissipation of the quadratic term]
\label{le: dissipation quadratic SC}
Let $x_k,v_k,y_k$ be defined by~\eqref{Sto_FE-SC-x},\eqref{Sto_FE-SC-v}, \eqref{yck} \eqref{wk_stock} then
\begin{align*}
\frac{\mu}{2}|v_{k+1} - x^* |^2 - \frac{\mu}{2}|v_{k} - x^* |^2  \leq &-h_k\sqrt{\mu}E^{SC}_k  -\langle \nabla f(y_k) , y_k -x_k \rangle\\
& + \left( \frac{\mu}{2}+ \frac{h_k\sqrt{\mu}L}{2} - \frac{\sqrt{\mu}}{2h_k} \right)|x_k - y_k|^2 + \frac{h_k^2}{2}|\nabla f(y_k)|^2\\
&- h_k\sqrt{\mu} \langle v_k -x^* +x_k -y_k, e_k \rangle + h_k^2 \left\langle \nabla f(y_k) +\frac{e_k}{2}, e_k \right\rangle.
\end{align*}
\end{lemma}

\begin{proof}
By $1$-smoothness of the quadratic term in $E^{SC}_k$, %
\begin{align*}
\frac{\mu}{2}\left( |v_{k+1} - x^*|^2 - |v_{k} - x^*|^2  \right) &  = \mu \langle v_k -x^* , v_{k+1}-v_{k} \rangle + \frac{\mu}{2} |v_{k+1}-v_{k} |^2 &&  \text{by $1$-smoothness} \\
&=  -\mu w_k \langle v_k -x^* , v_k -x_k \rangle && \text{apply $v_k$ equation} \\
& - h_k\sqrt{\mu} \langle v_k -x^* , \nabla f(y_k) \rangle - h_k\sqrt{\mu} \langle v_k -x^* ,e_k \rangle\\
& + \frac{\mu}{2} |v_{k+1}-v_{k} |^2.
\end{align*}

We will deal with this quadratic below. First, by \eqref{yck}  we have 
\begin{equation}
\label{eq: xyv 1 sc}
	w_k(v_k - x_k) = \frac{w_k}{1-w_k}(v_k -y_k) = h_k\sqrt{\mu}(v_k -y_k), \qquad \text{from \eqref{wk_inv}}
\end{equation}
and
\begin{equation}
\label{eq: xyv 2 sc}
	v_k -y_k = \frac{1-w_k}{w_k}(y_k -x_k) = \frac{1}{h_k\sqrt{\mu}}(y_k -x_k), \qquad \text{from \eqref{wk_inv}}
\end{equation} 
This gives
\begin{align*}
\frac{\mu}{2}\left( |v_{k+1} - x^*|^2 - |v_{k} - x^*|^2  \right) 
 = & -h_k\mu \sqrt{\mu}\langle v_k -x^* , v_k -y_k \rangle && \text{ by \eqref{eq: xyv 1 sc}}\\
&  - h_k\sqrt{\mu}\langle v_k - y_k , \nabla f(y_k) \rangle  - h_k\sqrt{\mu}\langle y_k - x^* , \nabla f(y_k) \rangle && \text{add and subtract $y_k$} \\
& +\frac{\mu}{2} |v_{k+1}-v_{k} |^2 -h_k\sqrt{\mu} \langle v_k -x^* ,e_k \rangle.\\
\end{align*}
Note that for all $a,b,c \in \mathbb{R}^n$, the quadratic norm satisfies
\[
\langle a - b , a-c \rangle = -\frac{1}{2} |a - b |^2 -\frac{1}{2} |a - c |^2 +\frac{1}{2} |b-c |^2 ,
\]
take $a =v_k$, $b=x^*$ and $c = y_k$ to obtain
\begin{equation}
\label{eq: equality quadratic}
-\langle v_k -x^* , v_k -y_k \rangle = \frac{1}{2}|v_k -x^*|^2 + \frac{1}{2}|v_k - y_k|^2 -\frac{1}{2}|y_k -x^*|^2
\end{equation}
Then,
\begin{align*}
\frac{\mu}{2}\left( |v_{k+1} - x^*|^2 - |v_{k} - x^*|^2  \right) \leq &  -\frac{h_k\mu \sqrt{\mu}}{2}\left( |v_k -x^*|^2 + |v_k - y_k|^2 -|y_k -x^*|^2 \right) && \text{by \eqref{eq: equality quadratic}}\\
& - \langle y_k - x_k , \nabla f(y_k) \rangle -h_k\sqrt{\mu} \left( f(y_k) - f^* +\frac{\mu}{2}|y_k -x^*|^2 \right) && \text{by $\mu$-convexity}\\
&+\frac{\mu}{2} |v_{k+1}-v_{k} |^2 -h_k\sqrt{\mu} \langle v_k -x^* ,e_k \rangle.
\end{align*}
Then rearranging the terms, 
\begin{align}
\label{eq: dissip quadratic}
\frac{\mu}{2}\left( |v_{k+1} - x^*|^2 - |v_{k} - x^*|^2  \right) \leq &  - h_k \sqrt{\mu} \left( f(y_k) - f^* + \frac{\mu}{2}|v_k - x^*|^2\right) \nonumber\\
&  -\frac{h_k\mu \sqrt{\mu}}{2}|v_k - y_k|^2  - \langle y_k - x_k , \nabla f(y_k) \rangle \nonumber\\
 &+\frac{\mu}{2} |v_{k+1}-v_{k} |^2 -h_k\sqrt{\mu} \langle v_k -x^* ,e_k \rangle.
\end{align}

 In the first term in the right hand side, we would like $f(x_k)$ instead of $f(y_k)$. To obtain that, we study the quadratic term in the right hand side. Using the algorithm and expanding the square,

\begin{align*}
\frac{\mu}{2} |v_{k+1}-v_{k} |^2 = &  \frac{\mu}{2} \left( |w_k (x_k -v_k)|^2 -\frac{2h_k}{\sqrt{\mu}}\langle w_k (x_k -v_k), \nabla f(y_k) +e_k \rangle +\frac{h_k^2}{\mu}|\nabla f(y_k) +e_k|^2 \right) \\
 =& \frac{\mu}{2} \left( |y_k -x_k |^2 + \frac{2h_k}{\sqrt{\mu}}\langle y_k -x_k , \nabla f(y_k) +e_k \rangle +\frac{h_k^2}{\mu}|\nabla f(y_k) +e_k|^2 \right)&& \text{by \eqref{eq: xyv 1 sc}}\\
  = & \frac{\mu}{2}  |y_k -x_k |^2 +h_k \sqrt{\mu}\langle y_k -x_k , \nabla f(y_k)\rangle + \frac{h_k^2}{2}|\nabla f(y_k) |^2 && \text{expand the square}\\
 & + h_k \sqrt{\mu}\langle y_k -x_k , e_k \rangle + h_k^2\langle  \nabla f(y_k) +\frac{e_k}{2} , e_k\rangle .
\end{align*}
 Then using the $L$-smoothness of $f$, we obtain
\begin{align*}
\frac{\mu}{2} |v_{k+1}-v_{k} |^2  \leq & \frac{\mu}{2}  |y_k -x_k |^2 +h_k\sqrt{\mu} \left( f(y_k) - f(x_k) +\frac{L}{2}|y_k -x_k|^2 \right)  + \frac{h_k^2}{2}|\nabla f(y_k) |^2\\
 & + h_k \sqrt{\mu}\langle y_k -x_k , e_k \rangle + h_k^2\langle  \nabla f(y_k) +\frac{e_k}{2} , e_k\rangle\\
  \leq & h_k\sqrt{\mu} ( f(y_k) - f(x_k) ) + \left( \frac{\mu}{2} + \frac{L\sqrt{\mu}h_k}{2}\right)|y_k -x_k|^2  + \frac{h_k^2}{2}|\nabla f(y_k) |^2\\
 & + h_k \sqrt{\mu}\langle y_k -x_k , e_k \rangle + h_k^2\langle  \nabla f(y_k) +\frac{e_k}{2} , e_k\rangle,
\end{align*} 
 rearranging all the terms. Finally, inserting this inequality into \eqref{eq: dissip quadratic}, we get
\begin{align*}
\frac{\mu}{2}\left( |v_{k+1} - x^*|^2 - |v_{k} - x^*|^2  \right) \leq  & - h_k \sqrt{\mu} \left( f(y_k) - f^* + \frac{\mu}{2}|v_k - x^*|^2\right) +h_k\sqrt{\mu} ( f(y_k) - f(x_k) )  \\
&  -\frac{h_k\mu \sqrt{\mu}}{2}|v_k - y_k|^2 +  \left( \frac{\mu}{2} + \frac{L\sqrt{\mu}h_k}{2}\right)|y_k -x_k|^2 \\
& - \langle y_k - x_k , \nabla f(y_k) \rangle + \frac{h_k^2}{2}|\nabla f(y_k) |^2\\
 &-h_k\sqrt{\mu} \langle v_k -x^* ,e_k \rangle+ h_k \sqrt{\mu}\langle y_k -x_k , e_k \rangle + h_k^2\langle  \nabla f(y_k) +\frac{e_k}{2} , e_k\rangle.
\end{align*}
This concludes the proof using \eqref{eq: xyv 2 sc}.
\end{proof}

\begin{proof}[Proof of Proposition \ref{thm:Acc GD SC}]
From Lemma \ref{le: dissipation f SC} and Lemma \ref{le: dissipation quadratic SC}, the dissipation of $E_k^{SC} $ is 
\begin{eqnarray*}
E_{k+1}^{SC}  - E_k^{SC} 
& \leq & - h_k\sqrt{\mu} E_k^{SC} +h_k \left( h_k - \frac{1}{\sqrt{L}} \right) |\nabla f(y_k) |^2 + \frac{h_k\sqrt{\mu}}{2} \left( L - \frac{1}{h_k^2}\right) |x_k -y_k|^2   \\
& & -h_k \langle \sqrt{\mu}(x_k -y_k + v_k -x^*) -\frac{1}{\sqrt{L}}\nabla f(y_k), e_k \rangle +2h_k^2\langle  \nabla f(y_k) +\frac{e_k}{2} , e_k\rangle,
\end{eqnarray*}
Then, since $h_k \leq \frac{1}{\sqrt{L}}$, we have \eqref{eq: dissipation lip sto str convex}. To conclude the proof, remark that 
\[
\mathbb{E}[\beta_k] =  2h_k \left\langle \nabla f(y_k) , \mathbb{E}[e_k] \right\rangle +h_k\mathbb{E}[|e_{k}|^2]  -\left\langle \sqrt{\mu}(x_k - y_k + v_k -x^*) -\frac{1}{\sqrt{L}}\nabla f(y_k),\mathbb{E}[e_k]\right\rangle. 
\]
Since $\mathbb{E}[e_k] =0$ and $\mathbb{E}[|e_{k}|^2]  \leq \sigma^2$, \eqref{ass: mean and varience error}, $\mathbb{E}[\beta_k] \leq h_k \sigma^2$ which concludes the proof of Proposition \ref{thm:Acc GD SC}.
\end{proof}

\subsection{Convergence rates for the last iterate}
For a constant learning rate $h$, we recover the accelerated exponential convergence to a neighborhood of the minimum.  
\begin{proposition}
\label{prop:convergence neighborhood constant learning rate}
Choose a constant learning rate $h \in \left(0, \frac{1}{\sqrt{L}} \right]$ for \eqref{Sto_FE-SC-x}, \eqref{Sto_FE-SC-v}, \eqref{yck}, \eqref{wk_stock}. Then, for all $k\geqslant 0$,
\[
\mathbb{E}[E_k^{SC}] \leq r^k E_0^{SC} + (1 - r^k ) h \frac{\sigma^2}{\sqrt{\mu}},
\qquad r = (1 - h \sqrt{\mu})
\]
\end{proposition}

\begin{proof}
From Proposition \ref{thm:Acc GD SC}, we have
\[
\mathbb{E}[E^{SC}_{k+1}] \leq (1- \sqrt{\mu} h)\mathbb{E}[E^{SC}_k]  + h^2\sigma^2.
\]
Note that $h^2\sigma^2 = h \sqrt{\mu}\cdot \frac{h\sigma^2}{\sqrt{\mu}}$. Subtracting $\frac{h\sigma^2}{\sqrt{\mu}}$ from both side, we get
\[
\mathbb{E}[E^{SC}_{k+1}] -\frac{h\sigma^2}{\sqrt{\mu}}\leq (1- \sqrt{\mu} h)\left( \mathbb{E}[E^{SC}_k]  - \frac{h\sigma^2}{\sqrt{\mu}} \right).
\]
Then, by induction,
\[
\mathbb{E}[E^{SC}_{k}] -\frac{h\sigma^2}{\sqrt{\mu}}\leq (1- \sqrt{\mu} h)^k\left( E^{SC}_0 -\frac{h\sigma^2}{\sqrt{\mu}} \right),
\]
which concludes the proof.
\end{proof}

From the result above, we observe that we can initialize the algorithm with a small number of constant learning rate steps, so that the following condition holds. 
\begin{equation}
\label{Ecrit}\tag{$E_0$}
	E_0^{SC} \leq E_{crit} =  \frac{2 \sigma^2}{\sqrt{\mu L}}
\end{equation}
This condition is necessary to initialize the induction, so that the initial learning rate is small enough.  In the following remark we estimate the number of steps needed.
\begin{remark}
If we run the constant learning rate algorithm with $h_k= \frac{1}{\sqrt{L}}$ then, from Proposition \ref{prop:convergence neighborhood constant learning rate},
\[
\mathbb{E}[E_k^{SC}] \leq \left(1 - 
\sqrt{\frac{\mu}{L}}\right)^k \left(E_0^{SC} -\frac{\sigma^2}{\sqrt{L\mu}} \right) + \frac{\sigma^2}{\sqrt{L\mu}}.
\]
So the maximum number of steps, $K$, needed in order for \eqref{Ecrit} to hold is
\[
K =  \frac{\log(E_{crit}) - \log (2E_0^{SC})}{\log (1- \sqrt{\mu/L}) }.
\]
\end{remark}
In the ``warm-start'' case, the special case where the value of $E_0$ is known and is below the critical value, then we initialize with the corresponding (smaller) learning rate, given by the formula which follows.  

\begin{proposition}
\label{prop: decrease expectation E_k str convex acc}
Assume that \eqref{Ecrit} holds, by taking $K$ steps of the constant learning rate algorithm, if necessary.  
Let $h_k$ be given by 
\begin{equation}\label{hkdefnk0}
	 h_k := \frac{2}{\sqrt{\mu} (k + k_0)}, 
\end{equation}
where $k_0 = \max\{2\sqrt{C}, \frac{4\sigma^2}{\mu E_0^{SC}}\}$.  
Then 
\[
\mathbb{E}[E^{SC}_k] \leq \frac{4\sigma^2}{\mu (k +k_0)}.
\]
\end{proposition}

\begin{remark}
Observe that the rate is bounded by 
\[ 
\frac{4\sigma^2}{ \mu ( k +2 \sqrt{C_f} ) }.
\]
\end{remark}

\begin{proof}
The proof of Proposition \ref{prop: decrease expectation E_k str convex acc} can be done by induction and is an adaptation of the one of \cite{bottou2018optimization} or \cite{Oberman_Prazeres}. 
Initialization: If $k_0 = 2 \sqrt{C_f}$, i.e.
\[
2 \sqrt{C_f} \geq \frac{4\sigma^2}{\mu E_0^{SC}} \Rightarrow E_0^{SC} = \frac{2 \sigma^2}{\sqrt{\mu L}},
\]
then, $h_k \leqslant h_0= \frac{2}{\sqrt{\mu}(2 \sqrt{C_f})} = \frac{1}{\sqrt{L}}$ and,
\[
E_0^{SC} \leq \frac{4\sigma^2}{\mu (k_0)} = \frac{4\sigma^2}{\sqrt{L \mu}} = E_{crit}, 
\]
which holds by assumption. On the other hand, if $k_0 =\frac{4\sigma^2}{\mu E_0^{SC}}$, i.e.
\[
2 \sqrt{C_f} \leq \frac{4\sigma^2}{\mu E_0^{SC}} \Leftrightarrow E_0^{SC} \leq \frac{2 \sigma^2}{\sqrt{\mu L}},
\]
then $h_k \leq h_0 = \frac{2 \mu E_0^{SC}}{4\sqrt{\mu}\sigma^2} = \frac{ \sqrt{\mu} E_0^{SC}}{2\sigma^2} \leq \frac{1}{\sqrt{L}}$ and 
\[
E_0^{SC} \leq \frac{4\sigma^2}{\mu (k_0)} = E_0^{SC}.
\] 
For $k \geqslant 0$, Proposition \ref{thm:Acc GD SC} gives
\[
\mathbb{E}[E^{SC}_{k+1}] \leq (1- \sqrt{\mu} h_k)\mathbb{E}[E^{SC}_k]  + h_k^2\sigma^2.
\]
By definition of $h_k$, and using the induction assumption
\begin{align*}
\mathbb{E}[E^{SC}_{k+1}] \leq & \left( 1 - \frac{2}{k + k_0} \right) \frac{4\sigma^2}{\mu (k + k_0)} + \frac{4 \sigma^2}{\mu (k+k_0)^2}\\
\leq & \frac{4 \sigma^2}{\mu(k + k_0)} -2  \frac{4 \sigma^2}{\mu(k + k_0)^2}+ \frac{4 \sigma^2}{\mu (k+k_0)^2}\\
\leq & \frac{4 \sigma^2}{\mu(k + k_0)} -  \frac{4 \sigma^2}{\mu(k + k_0)^2}\\
\leq & \frac{4 \sigma^2( k + k_0 -1)}{\mu(k + k_0)^2}.
\end{align*}
Observe that 
\[
(k + k_0 -1) (k + k_0 -1)  \leq (k +k_0)^2 \Rightarrow \frac{ k + k_0 -1}{(k + k_0)^2} \leq \frac{1}{k+k_0 +1},
\]
which gives
\[
\mathbb{E}[E^{SC}_{k+1}] \leq \frac{4 \sigma^2}{\mu(k +1+ k_0)},
\]
and concludes the proof.
\end{proof}


\section{Derivation A-GD and A-SGD: Convex case}
\label{section: derivation Nest C}

\subsection{Derivation of accelerated algorithm}
In the case of convex functions, we have a similar, but different derivation of Nesterov's accelerated gradient descent. 
As before, the starting point is to couple gradient descent, \eqref{eq: Gradient Descent} with a faster doubled variable.  
Again, to avoid the evaluation of the gradient at two different points, we reduce to a single gradient evaluation at $y_k$, given as a convex combination of $x_k$ and $v_k$, with weight $w_k$.  
\begin{align}
\label{ycc}
	y_k  & =   \left(1-w_k\right) x_k  + w_k v_k	
\end{align}
In this case we do not have a condition number to use as a multiplier for the auxiliary variable, so instead, we choose a time-dependent multiplier $1/w_k$.
The $x$ equation will be coupled as before, but using $w_k$. 
However, in this case, the coupling is asymmetric, there is no coupling term in the $v$ equation. 
\begin{align}
	\label{xprimew}
	x_{k+1} - x_k &= w_k (v_k - x_k) - \frac{1}{L} \nabla f(y_k)\\
	\label{vprimew}
    v_{k+1} - v_k &= - \frac{1}{w_k L}\nabla f(y_k)
\end{align}
This yields Nesterov's method when we choose $w_k$ appropriately. 
\begin{proposition}
\label{prop: equivalence Nesterov convex}
Set $w_k =\frac{2}{k+2}$, then \eqref{xprimew} \eqref{vprimew} along with \eqref{ycc} corresponds to Nesterov's method in the convex case, \eqref{genericN1}, \eqref{genericN2}, \eqref{AGD_Convex_params}.
\end{proposition}

\begin{proof}
By definition of $y_k$ in \eqref{ycc}, we obtain \eqref{genericN1} from \eqref{xprimew},
\[
x_{k+1} = y_k - \frac{1}{L}\nabla f(y_k)
\]
Next eliminate the variable $v$ in \eqref{vprimew} using  using \eqref{ycc} at $k$ and $k+1$, and eliminate the gradient term using \eqref{genericN1},  to obtain 
\begin{equation}
	\label{step2}
	\frac{1}{w_{k+1}} (y_{k+1} - (1-w_{k+1})x_{k+1})
 = \frac{1}{w_k}(y_{k} - (1-w_k)x_{k}) + \frac 1{w_k} (x_{k+1}-y_k)  
\end{equation}
Simplify, first noticing that the $y_k$ terms cancel:
\[ 
y_{k+1} = x_{k+1} + \frac{w_{k+1}}{w_k}(1-w_k)(x_{k+1}-x_k) 
\]
Making the choice $w_k  =\frac{2}{k+2}$ leads to $\beta_k = k/(k+3)$ as in \eqref{AGD_Convex_params}, as desired. 
\end{proof}

\subsection{Extension to the stochastic case: variable learning rate}
In order to derive the right algorithm in the stochastic case, we generalize from 
\eqref{ycc}, \eqref{xprimew} \eqref{vprimew} to allow for a variable learning rate. 


 Based on Proposition \ref{prop: equivalence Nesterov convex}, $w_k$ can be seen as
\[
w_k = \frac{2}{k+2} = \frac{2h_k}{h_k(k+2)} = \frac{2h_k}{t_k}, \qquad \text{ with } t_k = h_k(k+2),
\]
 with constant time step $h_k=\frac{1}{\sqrt{L}}$.

In the following, we introduce a variable time step $h_k$ which may varies between $0$ and $\frac{1}{\sqrt{L}}$. Given a time step/learning rate $h_k$ and a discrete time $t_k$, define 
\begin{equation}
	\label{wdefnHT}
	w_k = \frac{2h_k}{t_k}
\end{equation}

Then, using the \eqref{wdefnHT} in 
\eqref{xprimew} 
we have the system
\begin{equation}
\label{FEg}\tag{FE-C}
\left \{
\begin{aligned}
		x_{k+1} - x_k &=  \frac{2h_k}{t_k} (v_k - x_k) - \frac{h_k}{\sqrt{L}}\nabla f(y_k)\\
    v_{k+1} - v_k &= - \frac{h_k t_k}{2}\nabla f(y_k)
\end{aligned}
\right . 
\end{equation}	
along with \eqref{ycc}.  

In the next proposition, we eliminate the $v$ variable in order to rewrite the algorithm of the form \eqref{genericASGD1}-\eqref{genericASGD2}. 
\begin{proposition}
\label{prop: generealization Nesterov C}
Let $h_k \in \left(0, \frac{1}{\sqrt{L}} \right]$ and $t_k$ be a divergent sequence. Then \eqref{FEg} along with \eqref{wdefnHT} can be rewritten as  
\begin{align}
		x_{k+1} &= y_k  - \frac{\alpha_k}{L} \nabla f(y_k)\\
   y_{k+1} & = x_{k+1} + \beta_{k}(x_{k+1} - x_k) + \gamma_{k} ( y_k -x_k)
\end{align}
\begin{equation}
	\alpha_k = \sqrt{L} h_k
	\qquad
	\beta_k = \frac{h_{k+1}}{t_{k+1}}(t_k \sqrt{L} -1)  
		\qquad
	\gamma_k = \frac{t_k}{t_{k+1}} \frac{h_{k+1}}{h_k}(1 - \sqrt{L}h_k) 
\end{equation}
\end{proposition}

\begin{proof}
Write the last equation in \eqref{FEg} as 
\begin{equation}
\label{vht}
    v_{k+1} - v_k = - \frac{h_k \sqrt{L}}{w_k}\frac{h_k}{\sqrt{L}}\nabla f(y_k)
\end{equation}
Proceeding  as in Proposition \ref{prop: equivalence Nesterov convex}, \eqref{step2} becomes
\begin{equation}
	\frac{1}{w_{k+1}} (y_{k+1} - (1-w_{k+1})x_{k+1})
 = \frac{1}{w_k}(y_{k} - (1-w_k)x_{k}) + \frac {h_k \sqrt L}{w_k} (x_{k+1}-y_k)  
\end{equation}
	which leads to 
 \begin{align*}
  y_{k+1} = & x_{k+1}  -w_k x_{k+1} + \frac{w_{k+1}}{w_{k}} y_k  - \frac{w_{k+1}}{w_{k}}\left(1 - w_k \right)x_k + \frac{w_{k+1}}{w_{k}}  h_k\sqrt{L} (x_{k+1} - y_k)\\
   =& x_{k+1} + \frac{w_{k+1}}{w_{k}}\left( h_k\sqrt{L} - w_k \right)(x_{k+1} -x_k) + \frac{w_{k+1}}{w_{k}}\left( 1 - h_k\sqrt{L} \right)(y_k -x_k),
 \end{align*}
 
which gives the result.
\end{proof}


\subsection{Transition to continuous time} 

In order to perform our Lyapunov analysis, we introduce now a continuous interpretation of  \eqref{FEg}. System  \eqref{FEg} can be seen as the forward discretization of 
\begin{align}	
\label{NODE}\tag{1st-ODE}
\left \{
	\begin{aligned}
	\dot{x} &= \frac{2}{t}(v-x) -\frac{1}{\sqrt{L}}\nabla f(x)\\
	\dot{v} &= - \frac{t}{2}  \nabla f(x),	
	\end{aligned}
\right . 
\end{align}
when $h_k$ goes to $0$.

At least formally, the system \eqref{NODE} is equivalent to the following second order ODE
\begin{equation}\label{OurEqnCC}\tag{H-ODE}
	\ddot{x} + \frac{3}{t}\dot{x} + \nabla f(x)  = - \frac{1}{\sqrt{L}} \left (D^2 f(x) \cdot \dot{x} +  \frac{1}{t} \nabla f(x) \right) 
\end{equation}
which has an additional Hessian damping term with coefficient $1/{\sqrt{L}}$ compared to Su, Boyd and Cand\'es' ODE, \cite{su2014differential}, 
\begin{equation}\label{ODE_Su}\tag{A-ODE}
	\ddot{x} + \frac{3}{t}\dot{x} + \nabla f(x) =0 .
\end{equation}
In \cite{su2014differential}, the authors made a connection between Nesterov method \eqref{genericN1}-\eqref{genericN2}-\eqref{AGD_Convex_params} and \eqref{ODE_Su}.

%
%
%

Similarly to the strongly convex case, notice that \eqref{OurEqnCC} can be seen as the high-resolution equation from \cite{shi2018understanding} with the highest parameter value $\frac{1}{\sqrt{L}}$. Using a Lyapunov analysis, we will show that the same Lyapunov function of \eqref{ODE_Su} decreases faster along \eqref{NODE}. 
In addition, rewriting \eqref{OurEqnCC} as the first order system \eqref{NODE} allows us to extend the Lyapunov analysis in the stochastic case.

\section{Accelerated Stochastic algorithm: convex case}
\label{section: accelerated convex}

In this section we study the convergence of $x_k$ in \eqref{FEg} in the stochastic case, $\widehat{g}= \nabla f +e$. As in the strongly convex case, we start by evaluating the dissipation of a Lyapunov function of \eqref{NODE} during one step and then choosing an appropriate learning rate we obtain an accelerated rate of convergence for the last iterate.

\subsection{Dissipation of the Lyapunov function}
We want to use the stochastic approximation, $\hat{g}= \nabla f +e$, in the discretization of \eqref{NODE}, \eqref{FEg}, with a time step $h_k$, becomes
\begin{equation}
\label{FE_Stoc}\tag{Per-FE-C}
\left \{
\begin{aligned}
x_{k+1} - x_k &= \frac{2h_k}{t_{k}}(v_k-x_k) - \frac{h_k}{\sqrt{L}} (\nabla f(y_k) +e_k),\\
v_{k+1} - v_k &=  -  h_k\frac{t_{k}}{2} (\nabla f(y_k) +e_k),
\end{aligned}
\right. 
\end{equation}
where $y_k$ is as in \eqref{FEg}, $t_k = \sum_{i=0}^k h_i$.\\

\begin{definition}
Define the continuous time parametrized Lyapunov function
\begin{equation}\label{LiapCTS}
 E^{C}(t,x,v;\epsilon ):= (t-\epsilon)^2( f(x) - f^* ) +2|v - x^* |^2
\end{equation}	
Define the discrete time Lyapunov function $E_k^{C}$ by 
\begin{equation}
	\label{Ek_defn}
	E_k^{C} =  E^{C}(t_k,x_k,v_k ;h_k) =E^{C}(t_{k-1},x_k,v_k ;0)  
\end{equation}
By convention, $h_{-1}= t_{-1} =0$ which implies $E_0^C = 2|v_0 - x^* |^2$.
\end{definition}

It is well-known that $E^{C}$ is Lyapunov function for \eqref{NODE}, see \cite{su2014differential}. But note, compared to Su-Boyd-Cand\'es' ODE \eqref{ODE_Su}, there is a gap in the dissipation of the Lyapunov function $E^{C}$, which will not be there if the extra term, $-\frac{1}{\sqrt{L}}\nabla f(x)$, was missing.  In particular, if $z$ and $\tilde z$ are solutions of \eqref{ODE_Su} and \eqref{NODE} respectively, then we can prove faster convergence due to the gap.

\begin{proposition}
\label{prop: different gaps convex acc perturbed}
Let $x_k,v_k,y_k$ be sequences generated by \eqref{FE_Stoc}-\eqref{FEg}. Then, for $h_k \leq  \frac{1}{\sqrt{L}}$,
\begin{equation}
\label{eq:main inequality E convex acc perturbed}
E_{k+1} ^{C}- E_k^{C} \leq  h_k\beta_k,
\end{equation}
where $\beta_k  :=  -t_{k} \langle 2(v_k -x^*) - \frac{t_{k}}{\sqrt{L}} \nabla f(y_k) , e_k \rangle + 2h_k t_{k}^2 \left\langle \nabla f(y_k) +\frac{e_k}{2},e_k \right\rangle$.

In particular, we have
\begin{equation}
\label{eq:main inequality E convex expectation}
\mathbb{E}[E_{k+1}^{C}] \leq \mathbb{E}[E_k^{C}] +  h_k^2 t_k^2\sigma^2.
\end{equation}
\end{proposition}

Before proving Proposition \ref{prop: different gaps convex acc perturbed}, using the convexity and the $L$-smoothness of $f$, we explicit the dissipation of the first term in the Lyapunov function, $t_{k-1}^2(f(x_{k} - f^*)$. This is a generalization of the classical inequality obtained in \cite{attouch2016fast} or \cite{su2014differential} in the case $e_k=0$.

\begin{lemma}
\label{le: dissipation tf acc convex pert}
Assume $f$ convex and $L$-smooth. Then,
\begin{align}
\label{eq:dissipation E acc convex pert}
t_k^2(f(x_{k+1} - f^*) &- t_{k-1}^2(f(x_{k} - f^*) \nonumber\\
\leq & -h_k^2 (f(x_k) -f^*) + 2h_kt_k\langle \nabla f(y_k),v_k -x^*  \rangle  -\left(\frac{1}{\sqrt{L}} -\frac{h_k}{2}\right)h_k t_k^2|\nabla f(y_k) |^2 \nonumber\\
 &-\frac{h_k t_k^2}{\sqrt{L}}\langle \nabla f(y_k), e_k \rangle +h_k^2 t_k^2\left\langle \nabla f(y_k) +\frac{e_k}{2},e_k \right\rangle.
\end{align}
\end{lemma}

\begin{proof}
The proof is a perturbed version of the one in \cite{beck2009fast,su2014differential,attouch2016fast}. First we prove the following inequality: For all $x,y,z$, $f$ satisfies
\begin{equation}
\label{eq: ineq 3 points}
f(z) \leq f(x) +\langle \nabla f(y), z-x \rangle +\frac{L}{2}|z-y|^2.
\end{equation}

By $L$-smoothness, 
\[
f(z) - f(x) \leq f(y) - f(x)  + \langle \nabla f(y) , z-y \rangle +\frac{ L}{2}|z-y|^2.
\]
and since $f$ is convex, 
\[
f(y) - f(x) \leq \langle \nabla f(y) , y -x \rangle.
\]
\eqref{eq: ineq 3 points} is obtained combining these two inequalities.\\

Now apply inequality \eqref{eq: ineq 3 points} at $(x,y,z)=(x_k,y_k, x_{k+1})$:  
\begin{align*}
f(x_{k+1})  
 \leq &  f(x_k) +\langle \nabla f(y_k), x_{k+1}-x_k \rangle +\frac{L}{2}|x_{k+1}-y_k|^2\\
\leq  & f(x_k) +\frac{2h_k}{t_k}\langle \nabla f(y_k), v_k -x_k \rangle -\frac{h_k}{\sqrt{L}}|\nabla f(y_k) |^2   && \text{by \eqref{FE_Stoc}}\\
&-\frac{h_k}{\sqrt{L}}\langle \nabla f(y_k), e_k \rangle + \frac{h_k^2}{2}|\nabla f(y_k) +e_k|^2.
\end{align*}
Expanding the square, we obtain
\begin{align}
\label{eq: inq 1}
f(x_{k+1}) \leq & f(x_k) +\frac{2h_k}{t_k}\langle \nabla f(y_k), v_k -x_k \rangle -\left(\frac{1}{\sqrt{L}} -\frac{h_k}{2}\right)h_k |\nabla f(y_k) |^2 \\
& -\frac{h_k}{\sqrt{L}}\langle \nabla f(y_k), e_k \rangle + h_k^2\left\langle \nabla f(y_k) +\frac{e_k}{2},e_k \right\rangle.
\end{align}

If we apply \eqref{eq: ineq 3 points} also at $(x,y,z)=(x^*,y_k, x_{k+1})$ we obtain
\begin{align*}
f(x_{k+1})  \leq &  f^* +\langle \nabla f(y_k), x_{k+1}-x^* \rangle +\frac{L}{2}|x_{k+1}-y_k|^2\\
 \leq &  f^* +\langle \nabla f(y_k), y_k -x^* \rangle -\frac{h_k}{\sqrt{L}}|\nabla f(y_k) |^2 && \text{by \eqref{FE_Stoc}}\\
 & -\frac{h_k}{\sqrt{L}}\langle \nabla f(y_k), e_k \rangle + \frac{h_k^2}{2}|\nabla f(y_k) +e_k|^2,
\end{align*}
then, expanding the square
\begin{multline}
\label{eq: inq 2}
f(x_{k+1})  \leq  f^* +\langle \nabla f(y_k), y_k -x^* \rangle -\left(\frac{1}{\sqrt{L}} 
-\frac{h_k}{2}\right)h_k |\nabla f(y_k) |^2\\
-\frac{h_k}{\sqrt{L}}\langle \nabla f(y_k), e_k \rangle +h_k^2 \left\langle \nabla f(y_k) +\frac{e_k}{2},e_k \right\rangle.
\end{multline}

Summing $\left( 1 - \frac{2h_k}{t_k}\right)$\eqref{eq: inq 1} and $\frac{2h_k}{t_k}$\eqref{eq: inq 2}, we have
\begin{align*}
f(x_{k+1}) -f^*  \leq &   \left( 1- \frac{2h_k}{t_k} \right)(f(x_k) -f^*)  + \frac{2h_k}{t_k}\langle \nabla f(y_k),v_k -x^*  \rangle -\left(\frac{1}{\sqrt{L}} -\frac{h_k}{2}\right)h_k |\nabla f(y_k) |^2 \\
&    -\frac{h_k}{\sqrt{L}}\langle \nabla f(y_k), e_k \rangle  +h_k^2 \left\langle \nabla f(y_k) +\frac{e_k}{2},e_k \right\rangle.
\end{align*} 
Then,
\begin{align*}
t_k^2(f(x_{k+1}) -f^*) \leq &  \left( t_k- 2h_k \right)t_k(f(x_k) -f^*) + 2h_k t_k\langle \nabla f(y_k),v_k -x^*  \rangle  -\left(\frac{1}{\sqrt{L}} -\frac{h_k}{2}\right)h_k t_k^2|\nabla f(y_k) |^2 \\
&  -\frac{h_k t_k^2}{\sqrt{L}}\langle \nabla f(y_k), e_k \rangle 
+h_k^2 t_k^2\left\langle \nabla f(y_k) +\frac{e_k}{2},e_k \right\rangle.\\
\end{align*}
Observe that $\left( t_k- 2h_k \right)t_k = (t_{k-1} -h_k)(t_{k-1} +h_k) = t_{k-1}^2 - h_k^2$, which gives \eqref{eq:dissipation E acc convex pert}.
\end{proof}

\begin{proof}[Proof of Proposition \ref{prop: different gaps convex acc perturbed}]

By definition of $v_{k+1}$ in \eqref{FE_Stoc}, we have
\begin{align*}
2|v_{k+1} -x^* |^2 &-  2|v_{k} -x^* |^2 \\
=& -2h_k t_k \langle v_k -x^*, \nabla f(y_k) +e_k \rangle + \frac{h_k^2t_k^2}{2}|\nabla f(y_k) +e_k |^2 && \text{$1$-smoothness of the quadratic term}\\
 = & -2h_k t_k \langle v_k -x^*, \nabla f(y_k) + \frac{h_k^2t_k^2}{2}|\nabla f(y_k)|^2 \\
& -2h_k t_k \langle v_k -x^*, e_k \rangle  + h_k^2t_k^2\left\langle \nabla f(y_k) +\frac{e_k}{2},e_k \right\rangle && \text{expanding the square.}
\end{align*}
Therefore, combining it with Lemma \ref{le: dissipation tf acc convex pert}, we obtain
\begin{align*}
E_{k+1}^{C}  - E_k^{C}   
&\leq  -h_k^2 (f(x_k) -f^*) -\left(\frac{1}{\sqrt{L}} - h_k \right)h_k t_k^2|\nabla f(y_k) |^2 \\
&-2h_k t_k  \langle v_k -x^* - \frac{t_k}{\sqrt{L}}\nabla f(y_k) , e_k \rangle + 2h_k^2t_k^2 \left\langle \nabla f(y_k) +\frac{e_k}{2},e_k \right\rangle,
\end{align*}
and, since $h_k \leqslant \frac{1}{\sqrt{L}}$, \eqref{eq:main inequality E convex acc perturbed} is proved.

For the second part of the proposition, observe that 
\[
\mathbb{E}[\beta_k] = -t_{k} \langle 2(v_k -x^*) - \frac{t_{k}}{\sqrt{L}} \nabla f(y_k) , \mathbb{E}[e_k] \rangle + 2h_k t_{k}^2 \left\langle \nabla f(y_k) ,\mathbb{E}[e_k] \right\rangle + h_k t_{k}^2\mathbb{E}[|e_k|^2].
\]
Since $e_k$ satisfies \eqref{ass: mean and varience error}, $\mathbb{E}[\beta_k] = h_k t_{k}^2$ which concludes the proof.
\end{proof}

\subsection{Convergence rate for the last iterate}
\label{subsection: variable time step acc convex}
 Proposition \ref{prop: different gaps convex acc perturbed} gives
\[
\mathbb{E}[E^{C}_{k+1} ] - \mathbb{E}[E^{C}_{k}] \leq h_k^2t_k^2 \sigma^2.
\]
In addition, by definition of $E_k^C$, \eqref{Ek_defn},
\begin{equation}
\label{eq: bound below E^C}
\mathbb{E}[E^{C}_{k} ]  \geq t_{k-1}^2 \mathbb{E}[f(x_{k})-f^*].
\end{equation}

\begin{proposition}
\label{prop: decrease expectation E_k convex acc}
Assume $h_k := \frac{c}{(k+1)^{3/4}} \leq \frac{1}{\sqrt{L}}$ and $t_k =\sum_{i=0}^k h_i$, then for all $k\geqslant 0$, 
\begin{align*}
\mathbb{E}[f(x_k)] - f^* \leq \frac{\dfrac{1}{16c^2}E_0^C + c^2 \sigma^2(1 +\log(k))}{\sqrt{k} }
\end{align*}
\end{proposition}

\begin{proof}
Summing \eqref{eq:main inequality E convex expectation} from $0$ to $k-1$, we obrain
\[
 \mathbb{E}[E_k^C] \leq E_0^C + \sigma^2\sum_{i=0}^{k-1} h_i^2 t_i^2.
\]
Now we want to prove that $ \mathbb{E}[E_k^C]$ is bounded.

By definition of $t_i$, 
\begin{align*}
t_i = \sum_{j=0}^i \frac{c}{(i+1)^{3/4}} \leq & \sum_{j=0}^i \int_j^{j+1} \frac{c}{x^{3/4}} \,dx && \text{by comparison series-intergal}\\
\leq & \int_0^{i+1} \frac{c}{x^{3/4}} \,dx \\
\leq &  4c(i+1)^{1/4}.
\end{align*}
So $t_i^2 \leq 16 c^2 (i+1)^{1/2}$ and 
\begin{align*}
\sum_{i=0}^{k-1} h_i^2 t_i^2 \leq  16 c^4\sum_{i=0}^{k-1} \frac{(i+1)^{1/2}}{(i+1)^{3/2}}
\leq & 16 c^4\sum_{i=0}^{k-1} \frac{1}{i+1}\\
\leq & 16 c^4\left( 1 + \int_{i=1}^{k} \frac{1}{x} \right)&& \text{by comparison series-intergal}\\
\leq & 16 c^4 ( 1 +\ln(k)).
\end{align*} 
Then, combine this with \eqref{eq: bound below E^C} to obtain
\begin{equation}
\label{eq: inequality expectation t_k f}
t_{k-1}^2 \mathbb{E}[f(x_{k})-f^*] \leq E_0^C + 16 c^4 ( 1 +\ln(k))\sigma^2.
\end{equation}

Now, remark that
\begin{align*}
t_{k-1} =   \sum_{i=1}^{k} \frac{c}{i^{3/4}}
\geq  \int_1^{k+1} \frac{c}{x^{3/4}}\,dx 
\geq  4c ( (k+1)^{1/4} - 1) 
\geq  4c k^{1/4} .
\end{align*}
This implies that $t_{k-1}^2 \geq  16 c^2 \sqrt{k}$ and concludes the proof dividing by $t_{k-1}^2$ in \eqref{eq: inequality expectation t_k f}.
\end{proof}

\bibliographystyle{abbrvnat}
\bibliography{sections/AccelerationBib}

\end{document}